\documentclass[a4paper,11 pt]{amsart}

\usepackage{geometry} 
\usepackage{amssymb}  
\usepackage{mathtools}      
\usepackage{mathabx}        
\usepackage[bb=fourier,cal=euler,scr=rsfs]{mathalfa}	
\usepackage{enumitem}       
\usepackage[colorlinks=true,backref=page]{hyperref} 

\usepackage[ansinew]{inputenc}

\hypersetup{bookmarksdepth = 3} 
\numberwithin{equation}{section}     

\setlist[enumerate,1]{label={\upshape(\roman*)},ref=\roman*}
\setlist[enumerate,2]{label={\upshape(\alph*)},ref=\alph*}

\newtheorem{theorem}{Theorem}[section]

\newtheorem{corollary}[theorem]{Corollary}
\newtheorem{lemma}[theorem]{Lemma}
\newtheorem{proposition}[theorem]{Proposition}
\newtheorem{prop}[theorem]{Proposition}

\newtheorem{claim}[theorem]{Claim}

\newtheorem{thmintro}{Theorem}

\theoremstyle{definition}
\newtheorem{definition}[theorem]{Definition}
\newtheorem{remark}[theorem]{Remark}

\newcommand{\eps}{\varepsilon}

\newcommand{\wt}[1]{\widetilde{#1}}
\newcommand{\R}{\mathbb R}
\newcommand{\RR}{\R}
\newcommand{\ZZ}{\mathbb Z}

\newcommand{\what}[1]{\widehat{#1}}

\newcommand{\wwp}{\widetilde \Phi}

\newcommand{\cW}{\mathcal{W}}
\newcommand{\cF}{\mathcal{F}}

\newcommand{\cG}{\mathcal{G}}

\newcommand{\cL}{\mathcal{L}}

\newcommand{\cC}{\mathcal{C}}
\newcommand{\cO}{\mathcal{O}}

\newcommand{\mt}{\widetilde M}

\newcommand{\hht}{\wt h}

\newcommand{\tY}{\wt Y}
\newcommand{\wphi}{\wt{\phi_t}}

\newcommand{\cs}{\cW^{cs}}
\newcommand{\cu}{\cW^{cu}}

\newcommand{\wh}{\widetilde h}

\title[Accessibility and ergodicity]
{Accessibility and ergodicity for collapsed Anosov flows}

\author{Sergio R.\ Fenley} 
\address{Florida State University, Tallahassee, FL 32306}
\email{fenley@math.fsu.edu}

\author{Rafael Potrie} 
\address{Centro de Matem\'atica, Universidad de la Rep\'ublica, Uruguay}
\email{rpotrie@cmat.edu.uy}
\urladdr{http://www.cmat.edu.uy/~rpotrie/}

\thanks{S.F. was partially supported by Simons Foundation grant numbers 280429 and 637554. R. P. was partially supported by CSIC 618.}

\begin{document}

\begin{abstract}
{We consider a class of 
partially hyperbolic diffeomorphisms introduced in \cite{BFP} which is open and closed and contains 
all known examples. If in addition the diffeomorphism 
is non-wandering, then we show it 
is accessible unless it contains a $su$-torus. This implies that these systems are ergodic when they preserve volume, confirming a conjecture by Hertz-Hertz-Ures \cite[Conjecture 2.11]{CHHU} for this class of systems.}

\bigskip

\end{abstract}

\maketitle

\medskip
\medskip


\section{Introduction} 
In this paper we study the ergodicity problem for 3 dimensional partially hyperbolic diffeomorphisms. It has been shown in \cite{HHU-Ergo} that ergodicity (in fact the K-property) is abundant, establishing a conjecture of Pugh and Shub in this setting \cite{PS}. This lead to the belief that non-ergodic partially hyperbolic diffeomorphisms in dimension 3 can be described \cite{HHU-dim3} (see also \cite{CHHU}). 

Many recent works including our own, have tried to show that in certain manifolds or isotopy classes ergodicity holds for \emph{all} volume preserving partially hyperbolic diffeomorphisms (see \cite{HHU-dim3, HUres,GS,FP, HRHU}). Thanks to the reductions of \cite{GPS, BurnsWilkinson, HHU-Ergo}
 the problem boils down to the study of accessibility, and this has become a problem on its own (see \cite{HP-Nil, Hammerlindl, HS, FP} for general results about accessibility). 

Here, instead of fixing a manifold or isotopy class, we work on a specific class of partially hyperbolic diffeomorphisms that were introduced in \cite{BFP} motivated by \cite{BFFP,BFFP2,FP2}. 
They are called {\emph {collapsed Anosov flows}}.
Very roughly the dynamics of the partially hyperbolic
diffeomorphism is semiconjugated to a self
orbit equivalence of an Anosov flow, and the semiconjugacy
sends flow lines to curves tangent to the center bundle.
See section \ref{s.background} for a precise definition.
The class of partially hyperbolic diffeomorphisms which
are collapsed Anosov flows is 
an open and closed subset of all partially hyperbolic diffeomorphisms,
and in addition it includes all known examples in manifolds with non-solvable fundamental group. We note that this class is strictly larger than the ones previously known to satisfy \cite[Conjecture 2.11]{CHHU} as it contains the partially hyperbolic diffeomorphisms in the connected component of the examples constructed in \cite{BGP,BGHP}. 

Our main result is the following:

\begin{thmintro}\label{teo.main}
Let $f:M \to M$ be a collapsed Anosov flow of a closed 3-manifold $M$ whose fundamental group is not virtually solvable. Assume that the non-wandering set of $f$ is all of $M$. Then, $f$ is accessible. In particular, if $f$ is $C^{1+}$ and volume preserving it is a K-system (and consequently
ergodic and mixing). 
\end{thmintro}

Again, precise definitions are given in \S~\ref{s.background}. It is worth noting that in \cite{BGP, BGHP} the argument to get ergodic examples is perturbative, so not even the concrete examples constructed there were known to be ergodic (but it was known that an open and dense subset of the examples in a volume preserving neighborhood were). 
The ergodicity for specific examples follows from the results of this paper.

We note that the assumption that $M$ does not have solvable fundamental group is necessary since the time one map of a 
suspension of a toral automorphism is not accessible (and in fact contains a $su$-torus). We refer the reader to \cite{Hammerlindl} for a complete treatment of accessibility in the class of 3-manifolds with virtually solvable fundamental group. 

Theorem \ref{teo.main} is complementary to what was done in \cite{FP} but some results have intersection
(in \cite{FP} it is proved that a more restrictive class, that of \emph{discretized Anosov flows}, are always accessible without the non-wandering assumption). In fact, the proof of Theorem \ref{teo.main} uses \cite{FP} at some point. However, a slightly weaker version of Theorem \ref{teo.main} can be proven independently of \cite{FP} (see \S~\ref{ss.scaf}). 

In \cite{FP} accessibility and ergodicity are established unconditionally in certain manifolds or isotopy classes of diffeomorphisms. On the other hand, in \cite{FP2} we showed that every partially hyperbolic diffeomorphism in a hyperbolic 3-manifold is a collapsed Anosov flow, so this paper gives a different proof\footnote{The argument of \cite{FP} achieves accessibility in hyperbolic 3-manifolds in a shorter way as it uses much less of the classification. In particular to use the proof we present here to deduce ergodicity in hyperbolic 3-manifolds one should use the full classification of partially hyperbolic diffeomorphisms in hyperbolic $3$-manifolds given in \cite{FP2} which is way longer.} of some of the results of \cite{FP}. Note that in \cite{FP3} we plan to show that every partially hyperbolic diffeomorphism in a Seifert 3-manifold is a collapsed Anosov flow, so this paper will imply that \cite[Conjecture 2.11]{CHHU} also holds for Seifert manifolds extending \cite{HRHU,FP} where this was shown for certain isotopy classes. 

The results of this article emphasize what \cite{BFP} proposes: it is valuable to understand the dynamics of collapsed Anosov flows and separate it from the classification of general partially hyperbolic diffeomorphisms. The proof uses some fine properties of (topological) Anosov flows that had not been previously used in the study of partially hyperbolic dynamics (see \S~\ref{ss.topologicalAnosovflows}).

\section{Background}\label{s.background}
In this paper $M$ will always denote a closed 3-manifold. In this section we briefly recall some basic properties of Anosov flows, partially hyperbolic diffeomorphisms and collapsed Anosov flows. Except in \S~\ref{ss.topologicalAnosovflows} we only present definitions and quote results from elsewhere. In \S~\ref{ss.topologicalAnosovflows} we prove Proposition \ref{p.periodiclozenge} which is an observation about self orbit equivalences of topological Anosov flows that we will use in the proof of Theorem \ref{teo.main}.  We refer the reader to \cite{BFP,FP} for more detailed information. 

For simplicity and since Theorem \ref{teo.main} 
can be shown in finite lifts 
 we will assume without explicit mention that all objects are orientable and foliations are transversally orientable; we comment on this assumption in \S~\ref{ss.orient}. 

\subsection{Topological Anosov flows}\label{ss.topologicalAnosovflows}
A flow $\phi_t: M \to M$ generated by a continuous vector field $X$ is said to be a \emph{topological Anosov flow} if it is expansive and preserves a topological foliation (see \cite[\S 5]{BFP} for other equivalent definitions). A topological Anosov flow $\phi_t$ preserves two transverse foliations $\cF^{ws}$ and $\cF^{wu}$ with the property that orbits in $\cF^{ws}(x)$ approach the orbit of $x$ in the future while orbits in $\cF^{wu}(x)$ approach the orbit of $x$ in the past. Leaves of $\cF^{ws}$ and $\cF^{wu}$ are planes or cylinders\footnote{One could have M\"{o}bius bands but our standing assumption that everything is orientable excludes this.} and all cylinder leaves contain a unique periodic orbit.  We refer the reader to \cite{Barthelme} for generalities on Anosov flows in dimension 3.

We denote by $\wphi : \mt \to \mt$ the lift of the flow to the universal cover. A fundamental property for us is 
\cite{FenleyAnosov,BarbotFeuill} that the quotient $\cO_\phi$ of $\mt$ by the orbits of $\wphi$ is homeomorphic to $\RR^2$. 
In other words, the flow $\wphi$ in $\mt$ is {\emph{topologically}}
a product flow. However, geometrically this is rarely the case, and this will be a key fact (see Theorem \ref{teo.AF} below for a precise statement) that we will exploit to show accesibility. 

We call $\cO_\phi$ the \emph{orbit space} of $\phi$.  The lifts $\wt{\cF^{ws}}$ and $\wt{\cF^{wu}}$ are flow invariant and therefore
induce transverse one dimensional foliations $\cG^s$ and $\cG^u$ in the orbit space $\cO_\phi$. 
The fundamental group $\pi_1(M)$ preserves the collection of
flow lines, hence induces an action by homeomorphisms on $\cO_\phi$
which preserves the foliations $\cG^{s}$ and $\cG^{u}$.
Given an orbit $o \in \cO_\phi$ we call \emph{half leaf} 
of $\cG^{s}(o)$ (or $\cG^u(o)$) to a connected component of $\cG^s(o)\setminus \{o\}$ (or $\cG^u(o) \setminus \{o\}$). 

Given an orbit $o$ of $\wphi$ we view it as both an element or
point in $\cO_\phi$ and as a subset of $\mt$ consisting of all
points in the orbit.

Note that if for some $o \in \cO_\phi$ there is some non trivial
$\gamma \in \pi_1(M)$ such that $\gamma o = o$ then this means that $o$ corresponds to a periodic orbit $\alpha= \pi(o)$ of $\phi_t$. 
In fact this is an if and only if property.
We say that $\gamma$ acts \emph{increasingly} (resp. \emph{decreasingly}) on $o$ if $\gamma x = \wphi(x)$ for some $t > 0$ (resp. $t<0$) and some $x \in o$ (since $\gamma$ does not have fixed points in $\mt$ it is easy to see that this is independent on $x \in o$). Our orientability assumptions imply that if a deck transformation fixes some orbit, then it also fixes all the \emph{half leaves} of $\cG^s(o)$ and $\cG^u(o)$, i.e. the connected components of $\cG^s(o) \setminus \{o\}$, $\cG^u(o) \setminus \{o\}$.   

Let $\alpha_1, \alpha_2$ be periodic orbits of $\phi_t$.
We say that they are \emph{freely homotopic} 
if the unoriented curves $\alpha_1, \alpha_2$ are freely homotopic.
With our orientation conditions this is equivalent to saying that
there is a non trivial
deck transformation $\gamma \in \pi_1(M)$ and lifts $o_1, o_2$ of $\alpha_1, \alpha_2$ respectively such that $\gamma o_1 = o_1$ and $\gamma o_2 = o_2$ (note that we do not require $\gamma$ to act increasingly on both, or any of them). There are many subtleties with defining freely homotopic orbits, having to do with orientation on the orbits, taking powers of
orbits;  but we will not enter into them here. 
We refer the reader to \cite{Fen1998} for more details.

A key object for the main result here will be lozenges, see figure~\ref{fig1}.

\begin{definition}\label{defi.lozenge}
An open subset $\cL$ of the orbit space $\cO_\phi$ is said to be a \emph{lozenge} if there are two orbits $o_1, o_2$ called the \emph{corners} of the lozenge which verify that they have half leaves $A^s_1$ of $\cG^s(o_1)$ and $A^u_1$ of $\cG^u(o_1)$ which are disjoint from half leaves $A^s_2$ of $\cG^s(o_2)$ and $A^u_2$ of $\cG^u(o_2)$ and satisfy the following properties:
\begin{itemize}
\item A leaf of $\cG^s$ intersects $A^u_1$ if and only
if it intersects $A^u_2$. 
\item 
A leaf of $\cG^u$ intersects $A^s_1$ if and only if intersects 
$A^s_2$.
\item every point $p \in \cL$ satisfies that
$\cG^s(p)$ intersects $A^u_1$ and separates $o_1$ from $o_2$.
In addition $\cG^u(p)$ intersects $A^s_1$ and separates
$o_1$ from $o_2$.
\end{itemize}
The half leaves $A^s_1, A^s_2, A^u_1, A^u_2$ are called the \emph{sides} of the lozenge. 
\end{definition}

\begin{figure}[ht]
\begin{center}
\includegraphics[scale=0.58]{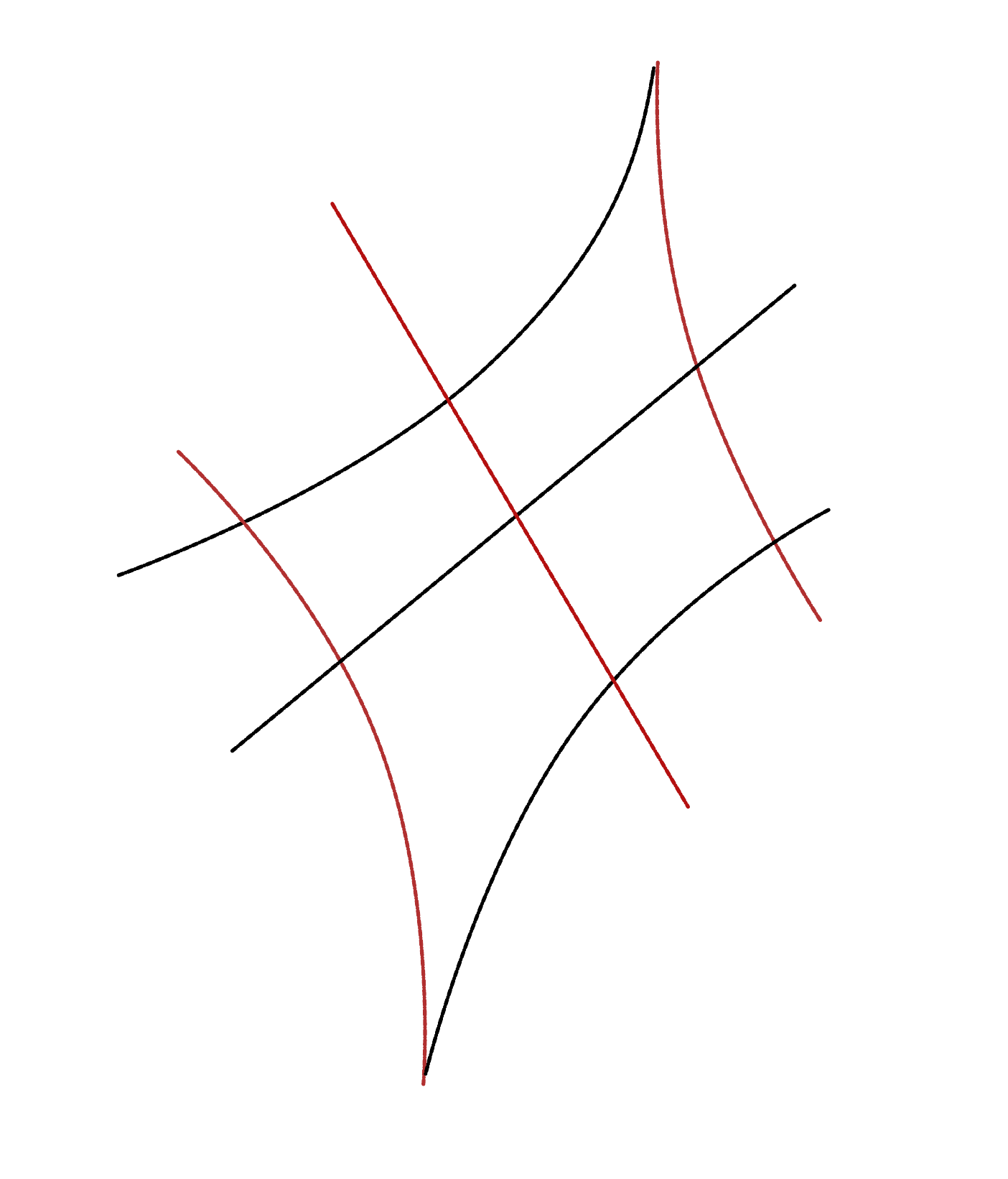}
\begin{picture}(0,0)
\put(-203,153){$o_1$}
\put(-130,153){$p$}
\put(-53,144){$o_2$}
\put(-160,52){$A_1^u$}
\put(-129,52){$A_2^s$}
\put(-87,77){{\small $\cG^u(p)$}}
\put(-196,91){{\small $\cG^s(p)$}}
\put(-113,225){$A_1^s$}
\put(-82,220){$A_2^u$}
\end{picture}
\end{center}
\vspace{-0.5cm}
\caption{{\small A lozenge. This figure is in the orbit space. So
topologically the region of $\mt$ which projects into the
lozenge is this lozenge times the reals.}}\label{fig1}
\end{figure}
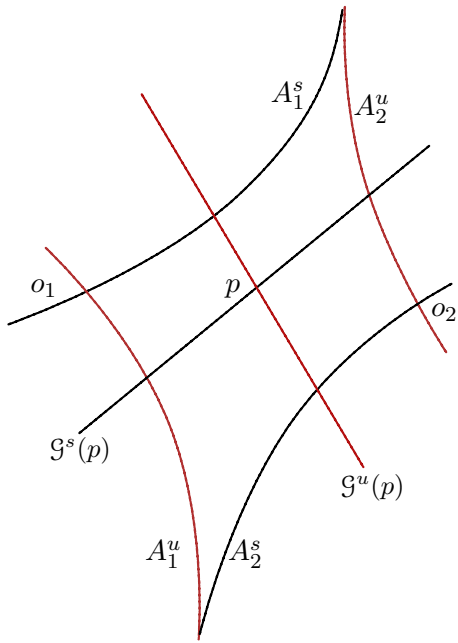

A \emph{chain of lozenges} is a sequence $\cL_1, \ldots, \cL_k$ of lozenges such that the closures of $\cL_i \cap \cL_{i+1}$ in
the orbit space intersect either at a side or at a corner.  
A chain is minimal if a given lozenge occurs at most once in
the chain $-$ no backtracking.

We will need the following result, the first claim of which is \cite[Theorem 3.3]{Fen95} and the second is \cite[Corollary E]{Fen1998}. For the last claim see \cite{Fen1998} or \cite[Theorem C]{BF-toroidalPA}.

\begin{theorem}\label{teo.AF}
Let $\phi_t: M \to M$ be a topological Anosov flow. Then, 
\begin{enumerate}
\item Let $o_1, o_2$ orbits in $\cO_\phi$ which are fixed by a 
non trivial $\gamma \in \pi_1(M)$. Then $o_1$ and $o_2$ are connected by a finite chain of lozenges whose corners are fixed by $\gamma$. 
\item If $\phi_t$ is not a suspension then there are freely homotopic 
periodic orbits of $\phi_t$.
\end{enumerate} 
Moreover, every $\ZZ^2$ subgroup of the fundamental group is associated to at most two bi-infinite minimal chains of lozenges. 
The subgroup fixes each such chain of lozenges.
There are infinitely many non trivial elements in the $\ZZ^2$
subgroup which fix every lozenge in these chains.
\end{theorem}

To get the last statement one uses that the $\ZZ^2$ subgroup
acts on the linearly ordered set of lozenges in the bi-infinite
chain of lozenges, which is order isomorphic to $\ZZ$.

As a consequence one deduces that every topological Anosov flow $\phi_t: M \to M$ which is not a suspension contains a lozenge whose corners are lifts
of periodic orbits of $\phi_t$ to $\mt$. This is the crucial place where we will use that the fundamental group is not solvable since suspensions are very different from the rest of Anosov flows in this respect.

\begin{corollary}\label{coro.AF}
If $\phi_t: M \to M$ is a topological Anosov flow which is not a suspension, then there exists a lozenge $\cL$ in $\cO_\phi$ which is fixed by some 
non trivial $\gamma \in \pi_1(M)$. 
\end{corollary}

Note that a topological Anosov flow is a suspension if and only if the manifold has (virtually) solvable fundamental group \cite{Brunella}. 

An important property of lozenges that we will use is the following (see figures~\ref{fig2} and \ref{fig3}):

\begin{proposition}\label{p.lozenges}
Let $\cL \subset \cO_\phi$ be a lozenge in $\cO_\phi$ fixed by $\gamma \in \pi_1(M)$. Denote by $o_1, o_2$ the corners of $\cL$ which correspond to periodic orbits of $\phi_t$. Then, if $\gamma$ acts increasingly on $o_1$ then it acts decreasingly on $o_2$ and vice versa. 
\end{proposition}

\begin{proof}
Note that if $\gamma$ acts increasingly on $o_1$ then $\gamma$ expands
points in $\cG^s(o_1)$ and contracts points in $\cG^u(o_1)$. 
This is counterintuitive so we explain for the action on $\cG^s(o_1)$.
Let $L$ be stable leaf of $\wphi$ which projects to $\cG^s(o_1)$ in
$\cO_\phi$.
Then $\gamma$ acts as an isometry in $L$ and sends a point
$x$ in $o_1$ to $\wphi(x)$ where $t > 0$. Let $y$ be in
$L$ in a nearby orbit at distance $d_0$ from $x$.
Then $\gamma(y)$ is at distance $d_0$ from $\wphi(x)$ 
since $\gamma$ acts as an isometry on $L$.
But the orbit through $y$ is closer to $\wphi(x)$ than $y$ is
to $x$ because orbits in a stable leaf converge together.
So the orbit through $\gamma(y)$ is farther from $o_1$ than
the orbit through $y$ is, that is $\gamma$ acts as an expansion
on the set of orbits in $L$.

Let $A^s_i$ be the half leaves of
$\cG^s(o_i)$ which are in the boundary of the lozenge,
and $A^u_i$ the half leaves of $\cG^u(o_i)$ which are in the 
boundary of the lozenge.
We proved that $\gamma$ acts in an exanding manner
in $\cG^s(o_1)$. Let $q$ in $\cG^s(o_1)$. Then $\cG^u(q)$ intersects
$A^s_2$ and separates $o_1$ from $o_2$ by the lozenge property.
This implies that $\gamma$ contracts points in $\cG^s(o_2)$.

Similarly $\gamma$ contracts points in $\cG^u(o_1)$ and expands 
points in $\cG^u(o_2)$.
\end{proof}

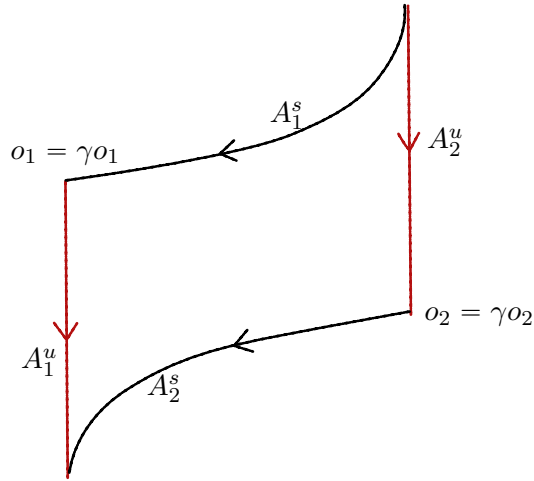
\begin{figure}[ht]
\begin{center}
\includegraphics[scale=0.69]{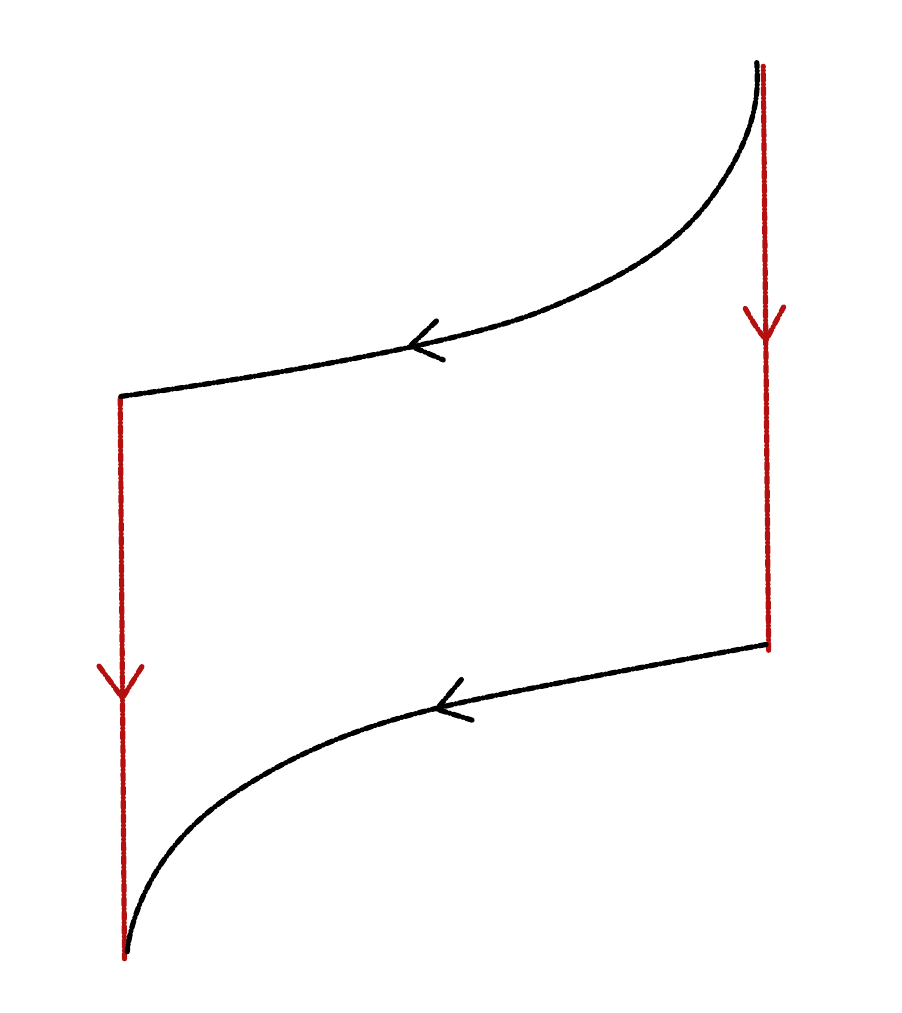}
\begin{picture}(0,0)
\put(-180,131){$o_1=\gamma o_1$}
\put(-25,70){$o_2=\gamma o_2$}
\put(-176,52){$A_1^u$}
\put(-129,42){$A_2^s$}
\put(-83,145){$A_1^s$}
\put(-24,135){$A_2^u$}
\end{picture}
\end{center}
\vspace{-0.5cm}
\caption{{\small A periodic lozenge fixed by a non trivial deck transformation
$\gamma$ which also fixes the corners $o_1, o_2$ of the lozenge.
The arrows indicate the action of $\gamma$ on the stable and
unstable leaves of $o_i$. Notice again that this figure depicts
the situation in the orbit space which is two dimensional.
The figure depicts the situation that $\gamma(x) = \wwp_t(x)$
with $t < 0$ for $x$ in $o_1$.}}\label{fig2}
\end{figure}

We now clarify what we mean by the flow not being 
a geometric product. 
Consider the setup as in Proposition \ref{p.lozenges}.
Let $\alpha_i = \pi(o_i)$ be the corresponding periodic
orbits of $\phi_t$ in $M$ ($\pi: \mt \rightarrow M$ is
the universal covering map).
Proposition \ref{p.lozenges} implies that 
$\alpha_1, \alpha_2$ are freely homotopic to the inverses 
of each other.
In particular this implies that 
a positive ray of $o_1$ is a bounded
Hausdorff distance in $\mt$ from a \emph{negative} ray of 
$o_2$ and vice versa. Also a positive ray of 
$o_1$ is not a bounded distance from a positive
ray of $o_2$. 
Hence the flow $\wwp$ in $\mt$ is very far from being
a geometric product.
The prototype flow for this behavior is the geodesic flow on
a closed,  orientable hyperbolic surface. Here every periodic
orbit is freely homotopic to the inverse of another periodic
orbit (the same geodesic traversed in the opposite direction).

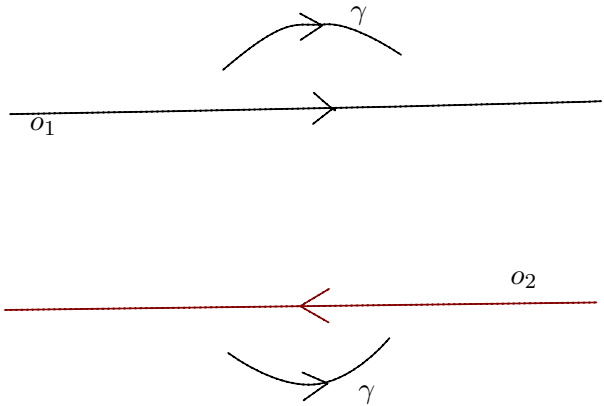
\begin{figure}[ht]
\begin{center}
\includegraphics[scale=0.48]{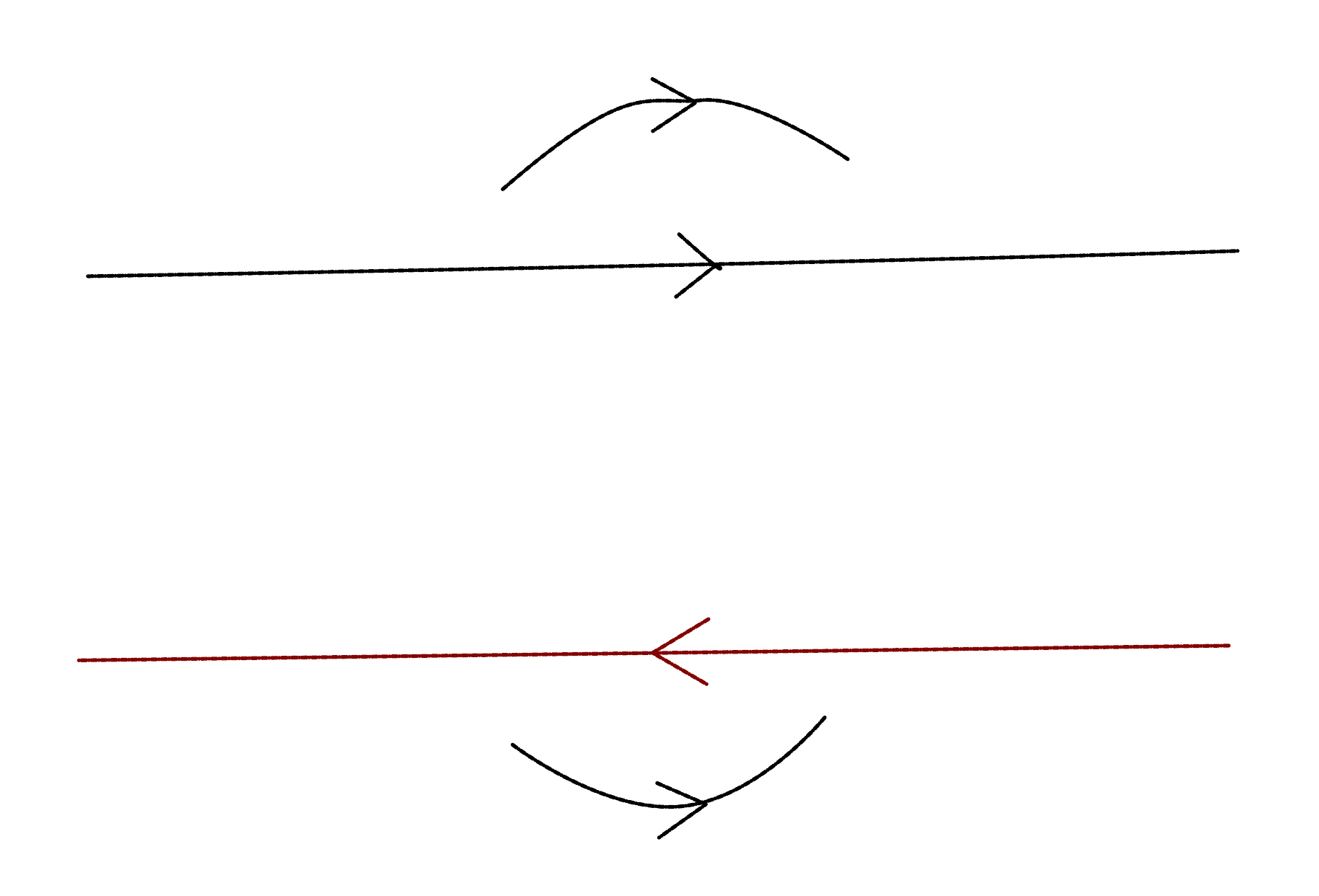}
\begin{picture}(0,0)
\put(-233,113){$o_1$}
\put(-53,55){$o_2$}
\put(-110,12){$\gamma$}
\put(-113,155){$\gamma$}
\end{picture}
\end{center}
\vspace{-0.5cm}
\caption{{\small Lifts $o_1, o_2$ of periodic orbits $\alpha_1, \alpha_2$ which
are freely homotopic to the inverses of each other. The
arrows in the flow lines 
indicate the positive flow direction in each
orbit. The action of $\gamma$ in each orbit is also
indicated. This figure
is supposed to be $3$-dimensional in $\mt$ and geometrically
correct in that positive rays of $o_1$ are a bounded distance
from negative rays of $o_2$ and so on. Hence geometrically the
flow is very far from being a product.}}\label{fig3}
\end{figure}

A \emph{self orbit equivalence} of an Anosov flow $\phi_t$ is a homeomorphism $\beta: M \to M$ sending orbits of $\phi_t$ to orbits of $\phi_t$ and preserving their orientation. See \cite{BaG,BFP} for more information on them.  
A self orbit equivalence is \emph{trivial} if it leaves
invariant every orbit.

We show the following proposition
about self orbit equivalences that will be an important technical
ingredient to prove our main result.
The proof and statement will assume some familiarity with the theory of Anosov flows (we note that very similar arguments can be found in \cite{BaG}; we include the proof since this is not stated explicitly). 

The proposition refers to $\RR$-covered Anosov flows. This means that
the leaf space of the stable foliation of the flow lifted to $\mt$
is homeomorphic to the reals. In case $\Phi$ is $\RR$-covered
and not conjugate to a suspension it has what is called
{\em skewed} type. The skewed type has a well defined structure,
in particular every orbit (periodic or not) is one corner of
a bi-infinite chain of lozenges. If the orbit is periodic then
there is a $\ZZ^2$ subgroup leaving this bi-infinite chain
invariant.  Any such flow has what is called a {\em one step up
map}: it is self orbit equivalence of the flow $\Phi$
which is homotopic to the identity and has a lift to $\mt$
 which induces a shift by one in any of these bi-infinite chains.
We refer the reader to \cite{FenleyAnosov} and
\cite{Barthelme} for details on $\RR$-covered Anosov flows,
including the one step up map.

\begin{proposition}\label{p.periodiclozenge}
Let $\beta: M \to M$ be a self orbit equivalence of a topological Anosov flow $\phi_t: M \to M$ on a manifold with non-virtually solvable fundamental group. Then, one of the following three possibilities happens:
\begin{enumerate}
\item there is a lozenge $\cL$ fixed by some non trivial
deck transformation $\gamma$ and there is 
a lift $\hat \beta$ of an iterate of $\beta$ such that $\hat \beta(\cL)=\cL$ 
\item $M$ is hyperbolic, $\phi_t$ is $\RR$-covered and $\beta$ is 
a non trivial  power of an one step up map, or
\item $M$ is Seifert and $\beta$ acts in the base as a pseudo-Anosov. 
\end{enumerate}
\end{proposition}

\begin{proof} 
If a 3-manifold $M$ admits an Anosov flow then it is irreducible
and the universal cover is homeomorphic to $\RR^3$.
It follows that  the manifold is either hyperbolic, Seifert fibered 
or it has a non-trivial JSJ decomposition (see e.g. \cite{Barthelme}).  

If $M$ has a non-trivial JSJ decomposition, then there is some iterate of $\beta$ that fixes all tori of the decomposition up to homotopy. Each torus 
in the torus decomposition is associated with a 
${\ZZ^2}$ subgroup of $\pi_1(M)$ which leaves 
invariant a minimal chain of lozenges $\cC$
(cf. Theorem \ref{teo.AF}). Let $\cL$ be a lozenge in $\cC$, which
has to be fixed by infinitely many $\gamma$ in the $\ZZ^2$ subgroup
associated with the torus.
Since an iterate of $\beta$ leaves the torus invariant up
to homotopy, this implies that 
 some lift $\hat \beta_0$ of an iterate of $\beta$
leaves invariant $\cC$. If $\beta_0$ does not fix a lozenge
in $\cC$ then it acts as a translation on the set of lozenges
in $\cC$. A $\ZZ^2$ subgroup leaving $\cC$ invariant
has a non trivial element $\gamma_1$
acting as translation on $\cC$, hence some iterate
$\hat \beta^i_0 \gamma^j_1$ with $i$ not zero fixes $\cL$,
and this is a lift $\hat \beta$ of a non trivial iterate
of $\beta$.

This shows that either the first possibility holds
or the manifold is hyperbolic or Seifert fibered. 

If $M$ is hyperbolic then up to an iterate we can assume that $\beta$ is homotopic to the identity. 
If $\beta$ has a finite iterate which is trivial, then the finite
iterate has a lift fixing all orbits in $\mt$. Since $M$ is hyperbolic
then $\phi_t$ is not conjugate to a suspension Anosov flow.
Then the first possibility of the proposition is satisfied automatically
because of item (ii) of Theorem \ref{teo.AF}.
If no power of $\beta$ is trivial, then
 in \cite{BaG} it is shown that $\phi_t$ must be $\RR$-covered and $\beta$ is 
a power of a one step up map. 


If $M$ is Seifert fibered, then the flow is topologically equivalent to a lift of a geodesic flow in a hyperbolic orbifold \cite{Brunella}, in particular, every periodic orbit corresponds to a closed curve in the base and all periodic orbits lifted to the universal cover are the corner
of some lozenge. If the action in the base of $\beta$ is not pseudo-Anosov, then one of these lozenges will be periodic.  

This finishes the proof of the proposition.
\end{proof} 

\subsection{Partially hyperbolic diffeomorphisms} 
A diffeomorphism $f: M^3 \to M^3$ is said to be partially hyperbolic if there exists a continuous $Df$-invariant splitting $TM = E^s \oplus E^c \oplus E^u$ into one dimensional bundles such that there exists $\ell>0$ so that for every $x \in M$ and $v^\sigma$ unit vectors in $E^\sigma_x$ ($\sigma=s,c,u$) we have: 

$$ \|Df^\ell v^s \| < \min \{1, \|Df^\ell v^c\| \} \leq \max  \{1, \|Df^\ell v^c\| \} < \|Df^\ell v^u \|. $$

It is well known that $E^s$ and $E^u$ integrate uniquely into $f$-invariant foliations $\cW^s$ and $\cW^u$ (see e.g. \cite{HP-survey}). In general, the center foliation does not integrate into a foliation, though one dimensionality allows to get some structures that we will not use explicitely here \cite{BI}. 

\begin{definition}
We say that $f$ is \emph{accessible} if for every $x,y\in M$ there exists a (piecewise smooth) curve tangent to $E^s \cup E^u$ joining $x$ to $y$.  
\end{definition}

There is a strong link between accessibility and ergodicity suggested by the
Pugh-Shub conjectures \cite{PS}. In our setting the following result is true (see \cite{BurnsWilkinson} for stronger results in higher dimensions that also imply the following): 

\begin{theorem}[Burns-Wilkinson]\label{teo.BW}
If $f: M \to M$ is a volume preserving partially hyperbolic diffeomorphism in a closed 3-manifold $M$ which is accessible and of class $C^{1+}$. Then, $f$ is a K-system, in particular it is ergodic and mixing.  
\end{theorem} 

Recall that a measure preserving system is a \emph{K-system} if every non-trivial finite partition has positive metric entropy, this implies the system is ergodic (i.e. invariant sets have zero or full measure). See \cite{BurnsWilkinson,CHHU} for more information. 

Lack of accessibility in dimension 3 allows to produce a strong structure that is what we will analyse here (see \cite{HHU-Ergo,HHU-dim3,CHHU}): 

\begin{theorem}[Hertz-Hertz-Ures]\label{teo-HHU}
Let $f: M \to M$ be a partially hyperbolic diffeomorphism on a closed 3-manifold with non-solvable fundamental group whose non-wandering set is all of $M$ and such that $f$ is not accessible. Then, there exists a lamination $\Lambda^{su}$ tangent to $E^s \oplus E^u$ such that: 
\begin{itemize}
\item  $\Lambda$ does not have closed leaves, 
\item the closure of the
complementary regions of $\Lambda^{su}$ (if existing) are $I$-bundles where $E^c$ is uniquely integrable and such that center curves form the $I$-bundle structure. 
\end{itemize}
\end{theorem}

See \cite[\S 2.3]{FP} for more discussion on this result. The case where the fundamental group is solvable is related with the existence of $su$-tori, see \cite{CHHU}.  

We will use the following easy property of partially hyperbolic diffeomorphisms that follows directly from uniform transversality of the bundles after iteration (see e.g. \cite[Proposition 4.2]{HP-survey}): 

\begin{proposition}\label{p.nocircles}
There are no closed curves tangent to $E^s$ or $E^u$. 
\end{proposition}

A generalization will be given in Proposition \ref{p.curvesgen} below. 

\subsection{Collapsed Anosov flows} 
We introduce here the notion of collapsed Anosov flows from \cite{BFP}. As mentioned earlier, it corresponds to an open and closed class of partially hyperbolic diffeomorphisms (see \cite[Theorem C]{BFP}) that contains all known examples of partially hyperbolic diffeomorphisms
in manifolds with non virtually solvable fundamental group (see \cite[Theorem A]{BFP}). 

\begin{definition}\label{defi.caf}
A partially hyperbolic diffeomorphism $f: M \to M$ is a \emph{collapsed Anosov flow} if there exists a (topological) Anosov flow $\phi_t: M \to M$, a self orbit equivalence $\beta: M \to M$ and a continuous map $h: M \to M$ homotopic to the identity  which is $C^1$ along orbits of the flow and such that $\partial_t h(\phi_t(x))|_{t=0} \in E^c(h(x)) \setminus \{0\}$ and such that $f \circ h  = h \circ \beta$. 
\end{definition}

In \cite[Theorem A]{BFP} we actually show that all known examples verify a slightly stronger assumption that we call \emph{strong collapsed Anosov flow} which under our orientability assumptions is also open and closed among partially hyperbolic diffeomorphisms. In \S \ref{ss.scaf} we introduce this notion since the proof of Theorem \ref{teo.main} admits a shortcut if one restricts to this class. (Note that it is an open question if being collapsed Anosov flow implies being strong collapsed Anosov flow.)

\subsection{Orientability assumptions}\label{ss.orient} \label{R: Added this subsection.} 
Here we comment on the assumption we have made that all bundles are orientable and that $Df$ preserves their orientation. This is no loss of generality as we shall now explain. 

Firstly, let us remark that if $f$ is a partially hyperbolic diffeomorphism of a manifold $M$ and $g$ is a lift of an iterate of $f$ to a finite cover $\hat M$ of $M$, then, accessibility of $g$ implies accessibility of $f$. To see this, notice that the strong stable and strong unstable manifolds of an iterate of $f$ are the same as those of $f$, so taking iterates does not change the fact that the whole manifold is an accessibility class. Further, an accessibility class in $\hat M$ projects to an accessibility class in $M$, thus, accessibility in $\hat M$ for $g$ implies accessibility of $f$ as desired.

Now we need to justify why taking a lift of an iterate of a collapsed Anosov flow to a finite cover is still a collapsed Anosov flow which will end the justification that our assumptions are in fact no loss of generality. 

First, notice that taking an iterate of a collapsed Anosov flow is still a collapsed Anosov flow: one just needs to keep the same flow $\phi_t$ and map $h$ in Definition~\ref{defi.caf} and pick an iterate of $\beta$. Consider now a collapsed Anosov flow $f : M \to M$ and let $\pi: \hat M \to M$ be a finite lift corresponding to the orientation of the bundles
(i.e. $\pi: \hat M \to M$ is a lift so that the bundles $E^s, E^c, E^u$ are orientable in the cover). In addition we
assume that $f$ lifts to $\hat f: \hat M \to \hat M$ so that $\hat f$
preserves all orientations of the bundles $-$ this can be achieved
by further lifts and iterates. We claim that $\hat f$ is also a collapsed Anosov flow. 
The flow $\phi_t$ lifts to 
$\hat \phi_t$, which is
 a topological Anosov flow in $\hat M$. In addition since $h$ is homotopic to the identity the
actions on homotopy induced by $f$ and $\beta$ on $\pi_1(M)$
are the same. Therefore
$\beta$ also lifts to $\hat \beta$ in 
$\hat M$. Clearly $\hat \beta$ is a self orbit equivalence
of $\hat \phi_t$. Finally the homotopy from the identity
to  $h$, lifts $h$ to $\hat h$ in $\hat M$. 
By construction $\hat f \circ \hat h = \hat h \circ \hat \beta$.
In other words $\hat f$ is also a collapsed Anosov flow.
This concludes the proof of our claim.

\section{Idea of the proof of Theorem \ref{teo.main}}

Here we explain some of the main ideas to prove Theorem 
\ref{teo.main}.

Since $f$ is a collapsed Anosov flow associated with
an Anosov flow $\phi_t$, there 
is a self orbit equivalence $\beta$ of the flow $\phi_t$,
and a map $h$
homotopic to the identity so that $f \circ h = h \circ \beta$.

Our assumptions give that there are periodic orbits
$\alpha_1, \alpha_2$ of $\phi_t$, which are
freely homotopic to the inverses 
of each other and 
 which lift to orbits $\widetilde \alpha_1, \widetilde \alpha_2$
that are 
corners of a lozenge $\cL$ invariant by deck transformation
$\gamma$ associated to $\alpha_1$ (c.f. Corollary \ref{coro.AF}).

By Proposition \ref{p.periodiclozenge} (and because the other cases where dealt in \cite{FP}) we can choose $\alpha_1, \alpha_2$
so that the lozenge $\cL$ is also invariant by a lift of a
power of the self orbit equivalence $\beta$.
Let $o_i = \widetilde \alpha_i$.

Roughly, the idea of the proof of Theorem \ref{teo.main} is
that the orbits $\alpha_1, \alpha_2$
above and the fact that they are (essentially)
transverse to the lamination $\Lambda^{su}$ force a
 closed curve in a leaf of $\Lambda^{su}$ which is invariant 
by a power of $f$. It is easy to show that this leads to a contradiction (see Proposition \ref{p.curvesgen}).

\begin{figure}[ht]
\begin{center}
\includegraphics[scale=0.38]{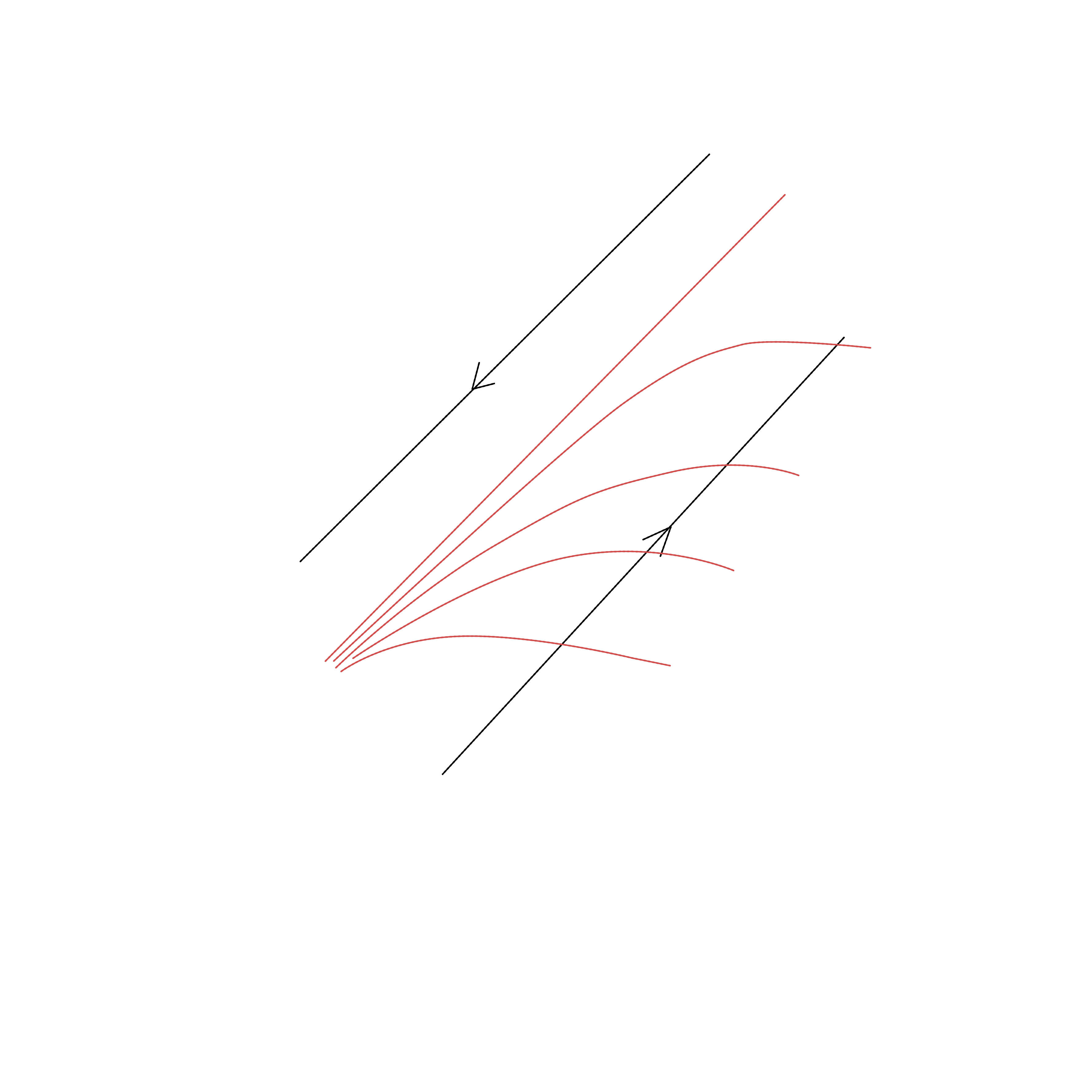}
\begin{picture}(0,0)
\put(-233,103){$o_2$}
\put(-157,13){$o_1$}
\put(-110,40){$E$}
\put(-87,75){$\gamma(E)$}
\put(-20,149){$\gamma^n(E)$}
\put(-110,168){$L$}
\end{picture}
\end{center}
\vspace{-0.5cm}
\caption{{\small $E$ is a leaf of $\widetilde \Lambda^{su}$ 
intersecting $o_1$. It cannot intersect $o_2$ and neither its
iterates by $\gamma$. Hence $\gamma^n(E)$ 
($n \rightarrow \infty$) converges to a 
unique leaf $L$ of $\widetilde \Lambda^{su}$
which separates $o_1$ from $o_2$
(the sequence $\gamma^n(E)$ 
may also converge to other leaves as well, but $L$ is 
uniquely determined). 
This figure is supposed to be 3-dimensional, we draw a 
2 dimensional slice of it.}}\label{fig4}
\end{figure}
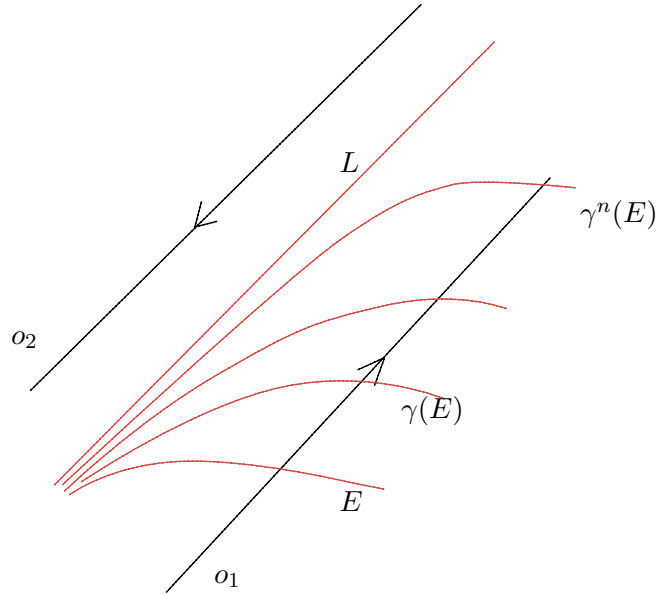

Let us explain how the curve is found: to simplify the explanation
let us assume in this section that $h$ is injective, hence a homeomorphism.
Notice that $h$ sends flow lines of $\phi_t$ to $C^1$ curves tangent
to $E^c$ and hence transverse to the lamination $\Lambda^{su}$.

Lift the homotopy of $h$ to the identity to produce
$\wh$.

We consider the action of $\gamma$ on $\wh(\widetilde \alpha_1)$
and $\wh(\widetilde \alpha_2)$. By Proposition \ref{p.lozenges},
$\gamma$ acts increasingly on $\wh(\widetilde \alpha_1)$
and decreasingly on $\wh(\widetilde \alpha_2)$. This is the 
crucial fact. In \S~\ref{s.transverselam} we show that this implies that for any
leaf $E$ of $\widetilde \Lambda^{su}$ then $E$ cannot intersect
both $\wh(\widetilde \alpha_1)$ and $\wh(\widetilde \alpha_2)$
because $E$ separates $\mt$.
Hence, starting with $E$ intersecting $o_1 = \widetilde \alpha_1$,
we can produce (see figure~\ref{fig4}) a unique leaf $L$ of $\widetilde \Lambda^{su}$
which separates in 
a specific way $\wh(\widetilde \alpha_1)$ from 
$\wh(\widetilde \alpha_2)$ (see Proposition \ref{p.lamination}). This leaf $L$ is fixed by $\gamma$.

In \S~\ref{s.surfacetransversetoAF} we do a careful analysis of how a surface like $L$ can intersect 
the image under $\wh$ of 
the stable and unstable leaves of $\widetilde \alpha_i$ and Proposition \ref{p.closedcurve} 
produces a curve in the intersection which is invariant by $\gamma$.
By the choice above we can prove this curve is also invariant
by a power of a lift of an iterate of $f$.
This is where the semiconjugacy betweeen $f$ and $\beta$
by $h$ is used.
This projects to a closed curve in a leaf of $\Lambda^{su}$
which is invariant by a power of $f$ that gives a contradiction as explained above.

In the rest of the article we carefully carry out this
strategy.

\section{Surfaces transverse to collapsed Anosov flows}\label{s.surfacetransversetoAF}

In this section we somewhat extend \cite[Proposition 4.3]{FenleySurfaces}. The context is slightly different since we need to take care of a different setup, but some of the ideas are very similar. Our result in this direction is Proposition \ref{p.boundary} below,  stating that in the orbit space of an Anosov flow, the boundary of the orbits that intersect a surface uniformly transverse to the flow is made of entire weak stable and weak unstable leaves.

The goal is to understand the intersection between a surface transverse 
to the center bundle of a (collapsed) Anosov flow and how it behaves in its boundary in the 
``collapsed orbit space". The main result of this section is Proposition \ref{p.closedcurve} which produces some closed curves in the intersection of certain surfaces transverse to the center bundle and the image under
$h$ of the weak stable or weak unstable foliation of the Anosov flow. This will be used later for certain $su$-surfaces to find some contradiction assuming non accessibility.

\subsection{Setup}\label{ss.setup} 
Let $\phi_t: M \to M$ be a (topological) Anosov flow and let $Y$ be a non-singular vector field in $M$ such that there is a map $h: M \to M$ continuous and homotopic to the identity, with $h$  being $C^1$ along orbits of the flow,
and in addition such that $Y(h(x))$ belongs to
$\RR_+ \partial_t h(\phi_t(x))|_{t=0}$.
By non singular we mean that $Y$ is continuous and never zero.

We consider in $\mt$, the universal cover of $M$, the orbit space $\cO_\phi$ of the flow $\wphi$ lifted to $\mt$. We can take $\hht$ a lift of $h$ commuting with deck transformations and at bounded distance from the identity (i.e. take a lift of a homotopy to the identity to construct $\hht$) and we denote by $\tY$ the lifted vector field.  

For $x \in \mt$ we denote $o_x$ to the orbit of $\wphi$ containing it. 

\begin{definition}Given an orbit $o$ of $\wphi$ we define $c_o$ to be the image by $\hht$ of $o$ which is a curve tangent to $\tY$. 
\end{definition}

We consider the sets $\what{ \cs}(c_o)$ and $\what{\cu}(c_o)$ to be the images by $\hht$ of the weak stable and weak unstable foliations of $o$. Note that these are a priori
only topological objects and not necessarily $C^1$ immersed surfaces.
In addition a priori there may be topological crossings, so the
collection of these may not 
form a (topological) branching foliation. 

\subsection{Transverse surfaces}\label{ss.transverse}
We consider $S$ a properly embedded surface in $\mt$ which is uniformly transverse to $\tY$. 
We will assume that for each $o \in \cO_\phi$ the curve $c_o$ intersects $S$ in at most one point\footnote{This is the case for instance, when $S$ separates $\mt$ in two connected components, e.g. if $S$ is the lift of a leaf of a Reebless foliation.}.

It will be important to understand better which curves
$c_o$ the surface $S$ intersects. Since $Y$ is only continuous, 
distinct curves $c_o$ may intersect. Hence it is easier and more
convenient to understand this set of intersections with $S$
from the point of view of the
orbit space of $\wphi$ as follows.
We define:

 $$S_\cO = \{ o \in \cO_\phi \ : \ c_o \cap S \neq \emptyset \}. $$

\begin{lemma}\label{l.Sopen}
The set $S_{\cO}$ is open in $\cO_\phi$. 
\end{lemma}
\begin{proof}
This is just transversality of $S$ and $\widetilde Y$ and that $\hht$ 
maps orbits of $\wwp$ to curves tangent to $\widetilde Y$. 
\end{proof}

We will define a function that indicates how the surface approaches $c_o$,
particularly when $o$ is a boundary point of $S_{\cO}$.
This function depends on some choices, but its asymptotic behavior for 
points in the boundary is well defined. Fix $x \in o$ and a transversal $D$ to $o$ at $x$. We can define $\tau^S: D \to \RR \cup \{\infty\}$ as 

$$\tau^S(y)=\infty \ \ \text{ if  } \ \ \hht(o_y) \cap S = \emptyset, \text{ and } $$

$$ \tau^S(y) = t_y \in \RR  \ \ \text{ if } \ \ \hht(\wt{\phi_{t_y}}(y)) \in S. $$

The funtion $\tau^S$ also depends on the choice of $D$ which
is left implicit.
Suppose that the orbit $o$ of $\wwp$ intersects $D$.
Then clearly $c_o$ intersects $S$ if and only if 
$\tau^S(o \cap D)$ is a real number, that is, it is finite.
 But in fact one deduces that $\tau^S$ is uniformly bounded in a 
neighborhood of 
$o \cap D$   when $c_o$ cuts $S$. 
Intuitively if $\tau^S(y) > 0$ then $S$ is ``above" (or flow forward of)
$\wh(y)$ and if $\tau^S(y) < 0$, then $S$ is ``below"
$\wh(y)$.
By uniform transversality, it is easy to see that: 

\begin{lemma}\label{l.taucontinuous}
The function $\tau^S$ is continuous in every point on which $\tau^S$ is finite. 
\end{lemma} 

\subsection{Boundaries} 
We want to understand how the function $\tau^S$ goes to infinity when
one approaches from 
$S_{\cO}$ the points of $\cO_\phi$ which are in the boundary of $S_{\cO}$.

\begin{lemma}\label{l.divergencetau} 
If $o \notin S_{\cO}$ then the function $\tau^S$ is uniformly bounded above in $\wt{\cF^{ws}}(o) \cap D$ and uniformly bounded below in $\wt{\cF^{wu}}(o)\cap D$. More precisely, for every $x \in o$ and transverse disk $D$ we have that there exists $K >0$ such that for every $z \in D$ such that $o_z \notin S_{\cO}$ then:
\begin{itemize}
\item  If $y \in \wt{\cF^{ws}}(o_z) \cap D$  
we have that $\tau^S(y) \in (-\infty, K) \cup \{\infty\}$.
\item  If
 $y \in \wt{\cF^{wu}}(o_z) \cap D$ then we have that
$\tau^S(y) \in (-K, +\infty) \cup \{\infty\}$.  
\end{itemize}
\end{lemma}
\begin{proof}
By uniform transversality between $S$ and the vector 
field $Y$, we know that if points $x,y \in \mt$ verify that $d(x,y)<\eps$ and $x \in S$ then every curve tangent to $Y$ through $y$ will intersect $S$.

We prove the first property as the second is entirely analogous.
Since $h$ is continuous there is $\delta > 0$ so that if
$d(x,y) < \delta$ then $d(\wh(x),\wh(y)) < \eps$.
By the contraction property of Anosov flows we know that there exists $t_0>0$ such that for every $z \in D$ we have that if $y \in \wt{\cF^{ws}}(o_z) \cap D$ and $t>t_0$ 
then $d(\wphi(y),\widetilde \phi_{\R}(z))< \delta$. It follows that if $o_z \notin S_{\cO}$ but 
$o_y \in S_{\cO}$, then $\tau^S(y) < t_0$.  Since the $t_0$ above
can be chosen uniform over $D$, this implies the result. 
\end{proof}

Recall that we denote by $\cG^s$ and $\cG^u$ to the foliations in $\cO_\phi$ induced by $\wt{\cF^{ws}}$ and $\wt{\cF^{wu}}$. As a consequence, we get: 

\begin{prop}\label{p.boundary}
If $o \in \partial S_{\cO}$ then either $\cG^s(o)$ or $\cG^u(o)$ is contained in $\partial S_{\cO}$. 
The first case happens if there is $V$ a neighborhood of $o \cap D$ in $D$ such that $\tau^S$ is bounded below in $V$ and the second case happens if $\tau^S$ is bounded above in a similar neighborhood.  
\end{prop} 

\begin{proof}
As in the previous setup,
pick $x \in o$ and a disk $D$  transverse to $\wphi$ through
$x$. Consider $x_n \in D$ so that $x_n \in D$ verify that $o_{x_n} \in S_{\cO}$, and $x_n$ converges to $x$.
In particular $\tau(x_n)$ is a real number for all $n$.
First notice that $|\tau^S(x_n)|$ converges  to $\infty$.
Otherwise up to subsequence it converges to a real number
and then by continuity $S$ intersects $\wh(o_x)$, contradiction.

Up to subsequence and without loss of generality we can assume that $\tau^S(x_n) \to +\infty$. 
Since $\tau(x_n)$ is a real number,
Lemma \ref{l.divergencetau} implies that for $n$ sufficiently
big and for every $y \in \wt{\cF^{ws}}(o_{x_n}) \cap D$ we have that $o_y \in S_{\cO}$. Therefore we get that given $z \in \wt{\cF^{ws}}(o) \cap D$ there is a sequence $z_n \in D$ such that $z_n \to z$ and $o_{z_n} \in S_{\cO}$. Moreover, $o_z \notin S_{\cO}$ since $\tau^S(z_n) \to +\infty$:
this is because if $o_z \in S_{\cO}$ then $\tau^S(z)$ is real and
uniformly bounded in a neighborhood of $z$ in $D$.  
This shows that the local weak stable of the manifold of $o$ is
contained in $\partial S_{\cO}$.

Now we iterate this analysis. 
For any $v$ in $\cG^s(o)$ we find $D_1 = D, D_2,...,  D_n$ 
small disks transverse to $\wphi$, consecutive ones intersecting
a common orbit of $\wphi$,
and so that the segment in $\cG^s(o)$ between $o$ and $v$
is contained in the union of the orbits of $\wphi$ through the $D_i$.
In the last paragraph we showed that for any 
$z \in \wt{\cF^{ws}}(o) \cap D_1$
we find 
$$z_n \in D \ \ {\rm with} \ \  z_n \rightarrow z, \ \ 
o_{z_n} \in 
S_{\cO}, \ \ {\rm and} \ \  \tau^s(z_n) \rightarrow +\infty.$$

\noindent
Choose $z$ so that
$o_z$ also intersects $D_2$ and now apply the result to $D_2$. 
We get that for every $y \in \wt{\cF^{ws}}(o) \cap D_2$ then 
$o_y \in \partial S_{\cO}$. Induction proves that
$\cG^s(o) \subset \partial S_{\cO}$.

In addition $S_{\cO}$ is connected and hence in this case 
intersects only one complementary component of $\cG^s(o)$ in 
$\cO_\phi$.

\vskip .08in
The remaining option in the analysis above is that
$\tau^S(x_n)$ converges to $-\infty$ as $x_n$ converges
to $x$. Then  the same
type of arguments show that $\cG^u(o) \subset \partial S_{\cO}$.
In addition for 
any $z \in \wt{\cF^{wu}}(o)$  and for $z_n$ in $D$ 
with $o_{z_n} \in S_{\cO}$ converging
to $z$ then 
$\tau^S(z_n)$ converges to $-\infty$.

By the complementary condition these two possibilities are
incompatible. 
This proves the proposition.
\end{proof}

\subsection{Additional invariance}\label{ss.aditional} 
We now assume in addition
that there is a nontrivial  $\gamma \in \pi_1(M)$ such that: 

\begin{enumerate}
\item $\gamma$ fixes a lozenge $\cL$ of $\wphi$ whose corners we denote by $e_1$ and $e_2$ and such that $\gamma$ acts increasingly in the orientation of $e_1$ and decreasingly in the orientation of $e_2$, 
\item $\gamma(S) = S$, 
\item $S$ does not intersect $c_1:= c_{e_1}=\hht(e_1)$ nor $c_2:= c_{e_2}=\hht(e_2)$, 
\item $S$ intersects some orbit of the lozenge, that is, there is $o \in \cL$ such that $\hht(o) \cap S \neq \emptyset$. 
\end{enumerate}

Our purpose is to show: 

\begin{prop}\label{p.closedcurve}
Under the above assumptions it follows that there is a $\gamma$-invariant curve in $\what{\cs}(c_1) \cap S$ or in $\what{\cu}(c_1) \cap S$.  Moreover, if a homeomorphism preserves $S$, \ $\what{\cs}(c_1)$ and $\what{\cu}(c_1)$ then it also preserves this $\gamma$-invariant curve.  
\end{prop}

The first thing we need to prove is that: 

\begin{lemma}\label{l.lozenge}
The set $S_{\cO}$ contains $\cL$. 
\end{lemma}
\begin{proof}
The set $S_{\cO}$ is an open connected set which intersects $\cL$ and it 
is $\gamma$-invariant, because 
$S, \cL$ and the sides are $\gamma$ invariant.

Assume that there is a point $o \in \cL \cap \partial S_{\cO}$.
Using Proposition \ref{p.boundary} assume without loss of generality that $\cG^u(o) \subset \partial S_{\cO}$. Since $o$ is in the interior of the lozenge it follows that 
$\cG^u(o) \cap \cG^s(e_i)$ for $i=1,2$. Since $S_{\cO}$ is $\gamma$-invariant and $\gamma$ acts as an expansion on $\cG^s(e_1)$ (resp $\gamma$ acts as a  contraction on $\cG^s(e_2)$) we get that $S_{\cO}$ must accumulate in both $e_1$ and $e_2$ from inside $\cL$. 

However, since $S_{\cO}$ is connected this is forbiden by the fact 
that $\cG^u(o) \subset \partial S_{\cO}$ and so $S_{\cO}$ cannot accumulate in both. This contradiction proves the lemma. 
\end{proof} 

We can now prove 

\begin{proof}[Proof of Proposition \ref{p.closedcurve}] 
Since $\cL \subset S_{\cO}$ and $e_1, e_2 \notin S_{\cO}$ by assumption we deduce that $e_1,e_2 \in \partial S_{\cO}$. 

We can then apply Proposition \ref{p.boundary} and without loss of generality we assume that $\cG^s(e_1) \subset \partial S_{\cO}$. We will find a curve in $\what{\cu}(c_1) \cap S$ which is invariant under $\gamma$. 

First, we note that since $e_1$ is in the boundary of $S_{\cO}$ and $\cG^s(e_1) \subset \partial S_{\cO}$ it follows that $\cG^u(e_1)$ is not
contained in $\partial S_{\cO}$. Since $\cL$ is contained 
in $S_{\cO}$ which is $\gamma$ invariant, then
it follows that there is a half leaf $A$ of $\cG^u(e_1)$ (i.e. a connected component of $\cG^u(e_1) \setminus \{e_1\}$) such that every $o \in A$ verifies that $\hht(o)$ intersects $S$. 
In other words $A \subset S_{\cO}$.

We know that both $A$ and $S$ are $\gamma$-invariant. Since for each orbit $o$ we have that the intersection point $S \cap \hht(o)$ is unique, we deduce that $S \cap \what{\cu}(c_1)$ is connected. This implies that it is a $\gamma$-invariant curve as desired and that if a homeomorphism preserves $S$ and $\what{\cu}(c_1)$ then it must preserve this curve. 
\end{proof}

\begin{remark}
The curve may auto-intersect both in $\mt$ as well as in its projection
in $M$. But the point is that its projection to $M$ is the image by a continuous map of a circle. 
This is because the curve is preserved by the non trivial
element $\gamma \in \pi_1(M)$.
\end{remark}

\section{Transverse laminations}\label{s.transverselam}
We continue in the setup of \S~\ref{ss.setup}. 

Let $\Lambda$ be a lamination transverse to $Y$ which verifies that: 

\begin{itemize}
\item $\Lambda$ does not have compact leaves, 
\item the closure of the
complementary regions of $\Lambda$ (if existing) are $I$-bundles where $Y$ is uniquely integrable and such that flowlines of $Y$ form the $I$-bundle structure. 
\end{itemize}

The standing assumption will be that the Anosov flow
$\phi_t$ is not a suspension (see
Corollary~\ref{coro.AF}). Under this assumption we can show that 
the assumptions in \S~\ref{ss.aditional} are verified for 
any surface $L$ in the lift of the lamination 
$\Lambda$ to the universal cover, as follows.
First notice that since $\Lambda$ is transverse to $Y$ and
$\Lambda$ is closed, then $\Lambda$ is uniformly transverse
to $Y$.
In addition, 
the second condition above implies that $\Lambda$ can be 
completed to a foliation $\cF$ so that it does not have
compact leaves and $\cF$ is transverse to $Y$ (this is well
known, but see explicit proofs and explanations
in \cite[Lemma 3.9]{FP}). Since 
$\cF$ does not have compact leaves it is Reebless,
therefore any curve in $\mt$ transverse to $\widetilde \cF$
intersects a leaf of $\widetilde \cF$ at most once. 
In particular any curve tangent to $\widetilde Y$ intersects
a leaf of $\widetilde \Lambda$ at most once as required in \S~\ref{ss.transverse}.

\begin{prop}\label{p.lamination}
For every lozenge $\cL$ fixed by a non trivial
 deck transformation $\gamma \in \pi_1(M)$ there is a  $\gamma$-invariant leaf $L$ of $\wt{\Lambda}$  so that $L$ intersects the image by $\hht$ of some orbit in the interior of the lozenge but does not intersect the image of the corner orbits of $\cL$.  
\end{prop}

\begin{proof}
Let $o_1, o_2$ be the corners of the lozenge $\cL$. Take a leaf $E \in \wt{\Lambda}$ intersecting $c_1= \hht(o_1)$. 
There is always such a leaf because $c_1$ is a properly embedded
 curve
tangent to $Y$ and the closures of complementary regions of $\Lambda$
are $I$-bundles where $Y$ is uniquely integrable.

\begin{claim}
The leaf $E$ cannot intersect $c_2=\hht(o_2)$. 
\end{claim}
\begin{proof}
Assume that $\gamma$ acts increasingly on $o_1$.
Proposition \ref{p.lozenges} implies that $\gamma$ acts decreasingly on $o_2$. Since $\hht$ commutes with deck transformations we get that the same happens in $c_1$ and $c_2$. 
Now we again use that $\Lambda$ can be completed to a Reebless
foliation. This implies that any leaf of $\widetilde \Lambda$
separates $\mt$. Since $\gamma$ acts increasingly in $o_1$ and
hence also in $c_1$, it follows that $\gamma(E)$ is on the
positive side of $E$ with respect to $Y$.
If $E$ intersects $c_2$ then since $\gamma$ acts decreasingly
on $o_2$, the same argument shows that $\gamma(E)$ is contained
in the negative side of $E$ with respect to $Y$. This is a contradiction
and proves the claim.
\end{proof}

Therefore, we can consider $V$ to be the region in $\mt$ between $E$ and $\gamma(E)$. 
We claim that $c_2$ cannot intersect $V$: $c_2$ is $\gamma$-invariant,
and if it  intersects $V$ then it must intersect $\gamma^{-1}(V)$ and thus intersect $E$ a contradiction. 

 Therefore, the open region 

$$R  \ = \ \bigcup_n \gamma^n(V \cup E)$$ 

\noindent
contains $c_1$ and is disjoint from $c_2$. The boundary of $R$ is accumulated by translates of $E$  under $\gamma^n, n \rightarrow \infty$,
therefore is saturated by leaves of $\wt{\Lambda}$. There must be a single leaf $L \in \wt{\Lambda}$ in the boundary
of $R$ that separates $c_1$ from $c_2$. By construction this leaf does
not intersect $c_1$ nor $c_2$ and is invariant by $\gamma$. 

We need to show that $L$ intersects the image of some orbit in the lozenge. But this is true because otherwise one could
connect $c_1$ and $c_2$ by a path with endpoints one in $c_1$ one
in $c_2$ and the interior a path which is the 
projection of a path in the lozenge $\cL$.
Hence $c_1, c_2$ would be 
in the same connected component of the complement of $L$. This completes the proof. 
\end{proof}

As a consequence of Proposition \ref{p.lamination} and Proposition \ref{p.closedcurve} we deduce:

\begin{corollary}\label{cor.gammainvcurve}
A lamination with the properties stated in the
beginning of this section verifies that there is a leaf $L \in \wt{\Lambda}$ invariant under a non trivial $\gamma \in \pi_1(M)$ and a $\gamma$-invariant orbit $o \in \cO_\phi$ such that $\what{\cs}(c_o) \cap L$ contains a curve which is $\gamma$-invariant. 
\end{corollary}

\section{Accessibility and ergodicity of partially hyperbolic diffeomorphisms} 
\subsection{Setup}\label{ss.setup2} 
We let $f: M \to M$ be a collapsed Anosov flow with respect to an Anosov flow $\phi_t: M \to M$ which is not a suspension.  There is a self orbit equivalence $\beta: M \to M$ and $h: M \to M$ a map homotopic to the identity, so that $h$ sends orbits of $\phi_t$ injectively onto curves tangent to $E^c$. In addition $f \circ h = h \circ \beta$.

In addition Remark \cite[Remark 2.6]{BFP} shows that for a collapsed
Anosov flow, the center bundle $E^c$ is orientable, and
it shows that $\partial_t h((\phi_t(x)) |_{t = 0}$ induces
an orientation on the center bundle $E^c$.
Therefore we choose 
$Y$ to be  the vector field of norm one such that 
$$Y(h(x)) \ \  \in \ \ \RR_{+}\partial_t h(\phi_t(x))|_{t=0}.$$ 
\noindent
The vector field $Y$ is contained in the $E^c$ bundle.
Note that this is the context of \S \ref{ss.setup}. 

To prove Theorem \ref{teo.main} we will assume that the non-wandering set of $f$ is all of $M$. This allows us to apply Theorem \ref{teo-HHU}. We will assume by contradiction that $f$ is not accessible so that Theorem \ref{teo-HHU} implies that there is a lamination $\Lambda^{su}$ tangent to $E^{s} \oplus E^u$ which satisfies the conditions of \S~\ref{s.transverselam}. 

As explained before, the proof of Theorem \ref{teo.main} is not affected if we take finite lifts and iterates, so we will assume for simplicity that all bundles of $f$ are orientable and their orientation is preserved by $Df$. We will reach a contradiction that will prove that $f$ is accessible, then the ergodicity part of Theorem \ref{teo.main} follows immediately from Theorem \ref{teo.BW}. 

\subsection{Strong collapsed Anosov flows}\label{ss.scaf}
We first give a direct proof under the extra assumption that $f$ is a \emph{strong collapsed Anosov flow} (see \cite{BFP} for discussions). This implies in particular that the map $h$ maps weak stable and weak unstable leaves of $\phi_t$ into surfaces tangent to $E^{cs}=E^s \oplus E^c$ and $E^{cu}= E^c \oplus E^u$ respectively. The reason we first show this case is that here we will not need to use the dynamics at all, and will get a direct contradiction to Proposition \ref{p.nocircles}. 

The contradiction follows from applying Corollary \ref{cor.gammainvcurve} which in this case produces a closed curve tangent to $E^s$ which is $\gamma$ invariant, contradicting Proposition \ref{p.nocircles}. 
This is because in this case $\what{\cs}$ is tangent to $E^{cs}$
and $\Lambda^{su}$ is tangent to $E^s \oplus E^u$, hence the
intersection is tangent to $E^s$.

\begin{remark}
As explained above, in principle, collapsed Anosov flows may always have this stronger property (that is, we currently know no example which
is a collapsed Anosov flow, but not a strong collapsed Anosov flow).
\end{remark}

\subsection{Collapsed Anosov flows}\label{ss.caf}
The proof of Theorem \ref{teo.main} is harder since it will need to appeal to \cite{FP} for some cases. However, for the other cases it is just slightly harder: instead of using Proposition \ref{p.nocircles} we will use some properties of self-orbit equivalences (cf. Proposition \ref{p.periodiclozenge}) to show that the curve given by Corollary \ref{cor.gammainvcurve} is also $\hat f$-invariant for some lift $\hat f$
of some iterate of $f$ which will be enough to get a contradiction. 

We now come back to the setting of \S~\ref{ss.setup2} and do not make the assumption on $h$ we did in the previous subsection (Subsection on
strong collapsed Anosov flows). 

\begin{lemma}\label{l.reduction1}
If $f$ is not accessible then there is some iterate of $\beta$ which has a lift $\hat \beta$ to $\mt$ which fixes some lozenge $\cL$ in $\cO_\phi$,
and so that $\cL$ is invariant by some non trivial $\gamma \in \pi_1(M)$..  
\end{lemma}
\begin{proof}
We apply Proposition \ref{p.periodiclozenge}. Unless the conclusion of the lemma is verified, we get that either $M$ is hyperbolic or $M$ is Seifert and $\beta$ acts as pseudo-Anosov in the base. Both cases where treated in \cite{FP} (see \cite[Theorem B]{FP} for the hyperbolic manifold case and \cite[Theorem D]{FP} for the Seifert case with base pseudo-Anosov). 
Specifically, in both cases, we proved in \cite{FP} that $f$
is accessible.
In the other cases we obtain a lozenge as in 
Proposition \ref{p.periodiclozenge}.
\end{proof}

As a consequence we deduce from Corollary \ref{cor.gammainvcurve} that there is a leaf $L \in \wt{\Lambda^{su}}$ which is invariant under the deck tansformation $\gamma$ as in the previous lemma,
as well as $L$ is invariant under 
some lift $\hat f: \mt \to \mt$ of an iterate of $f$. The proof of Theorem \ref{teo.main} is therefore completed with the following extension of Proposition \ref{p.nocircles}:

\begin{proposition}\label{p.curvesgen}
There are no closed curves (not necessarily injectively embedded) invariant under $f$ in an $f$ invariant  surface tangent to $E^{s} \oplus E^u$. 
\end{proposition}

\begin{proof}
If an $f$-invariant curve is not tangent to $E^s$, then iterates
of it have segments closer and closer to segments in unstable leaves.
Again since the curve is invariant one obtains that the curve
contains a segment which is contained in an unstable leaf.
But then invariance implies that the entire curve is contained
in an unstable leaf. This is impossible by Proposition \ref{p.nocircles}.
\end{proof}

\begin{remark}
Another proof of the previous proposition follows from the fact that one can construct a \emph{Lyapunov function} for the action of $f$ in the image of $\eta$ to show that $f$ is an expansive homeomorphism in the image of $\eta$ (see \cite[\S 2]{Potrie-Lew}) . Since the circle does not admit expansive homeomorphisms, that gives a contradiction. 
\end{remark}

{\small \emph{Acknowledgements:} The authors would like to thank Santiago Martinchich for discussions and his comments on the paper. }

\end{document}